\documentclass[a4paper,leqno,12pt]{amsart}
\usepackage[all]{xy}
\usepackage{xspace}
\usepackage{amsmath}
\usepackage{amstext}
\usepackage{amsfonts}
\usepackage[mathscr]{euscript}
\usepackage{amscd}
\usepackage{latexsym}
\usepackage{amssymb}
\usepackage{diagrams}
\theoremstyle{plain}
\swapnumbers
    \newtheorem{thm}{Theorem}[section]
    \newtheorem{prop}[thm]{Proposition}
    \newtheorem{lemma}[thm]{Lemma}
    \newtheorem{corollary}[thm]{Corollary}
    
    \newtheorem{subsec}[thm]{}
    \newtheorem*{thma}{Theorem A}
\theoremstyle{definition}
    \newtheorem{defn}[thm]{Definition}
    \newtheorem{example}[thm]{Example}

    \newtheorem{notation}[thm]{Notation}

\theoremstyle{remark}
        \newtheorem{remark}[thm]{Remark}

\newenvironment{myeq}[1][]
{\stepcounter{thm}\begin{equation}\tag{\thethm}{#1}}
{\end{equation}}

\newcommand{\mydiagram}[2][]
{\stepcounter{thm}\begin{equation}
     \tag{\thethm}{#1}\vcenter{\xymatrix{#2}}\end{equation}}
\newenvironment{mysubsection}[2][]
{\begin{subsec}\begin{upshape}\begin{bfseries}{#2.}
\end{bfseries}{#1}}
{\end{upshape}\end{subsec}}
\newcommand{\sect}{\setcounter{thm}{0}\section}
\newcommand{\wh}{\ -- \ }
\newcommand{\wwh}{-- \ }
\newcommand{\w}[2][ ]{\ \ensuremath{#2}{#1}\ }
\newcommand{\ww}[1]{\ \ensuremath{#1}}

\newcommand{\wb}[2][ ]{\ (\ensuremath{#2}){#1}\ }
\newcommand{\wref}[2][ ]{\ (\ref{#2}){#1}\ }
\newcommand{\hsp}{\hspace*{7 mm}}

\newcommand{\hsm}{\hspace*{2 mm}}
\newcommand{\vsn}{\vspace{1 mm}}

\newcommand{\vsm}{\vspace{3 mm}}
\newcommand{\hra}{\hookrightarrow}
\newcommand{\xra}[1]{\xrightarrow{#1}}
\newcommand{\xepic}[1]{\xrightarrow{#1}\hspace{-5 mm}\to}

\newcommand{\lra}[1]{\langle{#1}\rangle}

\newcommand{\epic}{\to\hspace{-3 mm}\to}

\newcommand{\rest}[1]{\lvert_{#1}}
\newcommand{\Aut}{\operatorname{Aut}}
\newcommand{\colim}{\operatorname{colim}}

\newcommand{\eval}{\operatorname{ev}}
\newcommand{\he}{\operatorname{h.e.}}
\newcommand{\ho}{\operatorname{ho}}

\newcommand{\hocolim}{\operatorname{hocolim}}

\newcommand{\Id}{\operatorname{Id}}
\newcommand{\lenG}{\operatorname{len}_{G}}
\newcommand{\Obj}{\operatorname{Obj}\,}
\newcommand{\op}{\sp{\operatorname{op}}}

\newcommand{\map}{\operatorname{Map}}

\newcommand{\hy}[2]{{#1}\text{-}{#2}}
\newcommand{\cA}{{\mathcal A}}
\newcommand{\AX}[1]{\Aut({#1})}
\newcommand{\AXp}[1]{\Aut\sb{\ast}({#1})}

\newcommand{\G}{\operatorname{haut}}
\newcommand{\GX}[1]{\G({#1})}
\newcommand{\hGX}[1]{\widehat{\G}({#1})}
\newcommand{\GXp}[1]{\G\sb{\ast}({#1})}

\newcommand{\PP}{\mathcal P}
\newcommand{\hPXY}[1]{\widehat{\PP}\sb{{#1}}}
\newcommand{\tPXY}[1]{\check{\PP}\sb{{#1}}}
\newcommand{\PXY}[1]{\PP\sb{{#1}}}
\newcommand{\cQ}{\mathcal Q}
\newcommand{\QXY}[2]{{\cQ}\sb{{#1},{#2}}}
\newcommand{\hQXY}[2]{\widehat{\cQ}\sb{{#1},{#2}}}

\newcommand{\cO}{{\EuScript O}}

\newcommand{\TT}{{\EuScript Top}}

\newcommand{\GT}{\hy{G}{\TT}}

\newcommand{\Ta}{\TT_{\ast}}
\newcommand{\hG}{\widehat{G}}

\newcommand{\hz}{\widehat{\zeta}}

\newcommand{\vre}{\varepsilon}

\newcommand{\hvr}{\widehat{\vre}}
\newcommand{\hdel}{\widehat{\delta}}
\newcommand{\hrho}{\widehat{\rho}}

\newcommand{\ZZ}{\mathbb Z}
\newcommand{\uX}{\underline{\mbox{X}}}
\newcommand{\ubX}[1]{\underline{\mbox{$\bX$}}\sb{#1}}

\newcommand{\BG}{\bB G}
\newcommand{\EG}{\bE G}

\newcommand{\EGX}[2]{\bE{#1}\times{#2}}
\newcommand{\EGXp}[2]{\bE{#1}\ltimes{#2}}

\newcommand{\WH}{W_{H}}

\newcommand{\XHH}{\bX\sp{H}_{H}}
\newcommand{\XpHH}{(\bX')\sp{H}_{H}}
\newcommand{\tXH}{\tX_{H}}
\newcommand{\tXHH}{\tX\sp{H}_{H}}
\newcommand{\uXHH}{\underline{\bX}\sp{H}_{H}}

\newcommand{\F}{\mathcal{F}}
\newcommand{\hF}{\widehat{\F}}

\newcommand{\OG}{\cO_{G}}

\newcommand{\TOG}{\TT\sp{\OG\op}}

\newcommand{\LG}{\Lambda\op}
\newcommand{\cj}[2]{\phi\sp{#1}\sb{#2}}

\newcommand{\cp}[2]{(\widetilde{\phi}\sp{#1}\sb{#2})\op}

\newcommand{\hf}{\widehat{f}}
\newcommand{\hg}{\widehat{g}}

\newcommand{\wtX}{{\raisebox{-1.5ex}{$\stackrel{\textstyle X}{\sim}$}}}
\newcommand{\bB}{\mathbf{B}}
\newcommand{\bC}{\mathbf{C}}
\newcommand{\bD}[1]{\mathbf{D}\sp{#1}}
\newcommand{\bE}{\mathbf{E}}

\newcommand{\bS}[1]{\mathbf{S}\sp{#1}}

\newcommand{\bX}{\mathbf{X}}
\newcommand{\Xp}{\bX_{1}}

\newcommand{\Xpp}{\bX'}

\newcommand{\Xs}{\bX\sb{\ast}}
\newcommand{\hX}{\widehat{\bX}}

\newcommand{\tX}{{\raisebox{-1.5ex}{$\stackrel{\textstyle \bX}{\sim}$}}}

\newcommand{\bY}{\mathbf{Y}}
\newcommand{\hY}{\widehat{\bY}}
\newcommand{\Ys}{\bY\sb{\ast}}
\newcommand{\bZ}{\mathbf{Z}}
\newcommand{\hZ}{\widehat{\bZ}}

\begin{document}
%
%
\title{Realizing homotopy group actions}
\author{David Blanc}
\author{Debasis Sen}
\address{Department of Mathematics\\ University of Haifa\\
31905 Haifa\\ Israel}
\email{blanc@math.haifa.ac.il,\ sen\_deba@math.haifa.ac.il}

\date{\today}


\begin{abstract}
For any finite group $G$, we define the notion of a
\emph{Bredon homotopy action} of $G$, modelled on the diagram of fixed point
sets \w{(\bX\sp{H})\sb{H\leq G}} for a $G$-space $\bX$, together
with a pointed homotopy action of the group \w{N\sb{G}H/H} on 
\w[.]{\bX\sp{H}/(\bigcup\sb{H<K} \bX\sp{K})}
We then describe a procedure for constructing a suitable diagram 
\w{\uX:\OG\op\to\TT} from this data, by solving a sequence of elementary 
lifting problems. If successful, we obtain a $G$-space \w{\bX'} realizing the 
given homotopy information, determined up to Bredon $G$-homotopy
type. Such lifting methods may also be used to understand other homotopy 
questions about group actions, such as transferring a $G$-action along a 
map \w[.]{f:\bX\to \bY}
\end{abstract}

\maketitle

\setcounter{section}{0}

%
%
\section*{Introduction}
\label{cint}

The naive notion of a homotopy action of a group $G$ on a topological space $\bX$ 
can be described as the choice of a homotopy class of a map
\w[,]{\BG\to\bB\GX{\bX}} where \w{\GX{\bX}} is the monoid of self-homotopy 
equivalences (see \S \ref{shaction}). This always lifts to a strict action, unique 
up to Borel equivalence (see \S \ref{dborel}). However, the $G$-actions 
we obtain in this way will be free, so the more delicate aspects of equivariant 
topology are not visible in this way.

A more informative approach to equivariant homotopy theory, due to Bredon, 
studies $G$-spaces $\bX$ up to $G$-homotopy equivalence \wh that is, $G$-maps 
having $G$-homotopy inverses (see \cite{BredE}). This is equivalent to the 
homotopy theory of diagrams \w{\uX:\OG\op\to\TT} (where \w{\OG} is the orbit 
category of $G$ and \w{\uX(G/H)} is the fixed point set \w{\bX\sp{H}} \wwh see 
\S \ref{socat} and \cite{ElmS}). Dwyer and Kan showed that this in turn is 
equivalent to a homotopy theory of a certain diagram of fibrations 
(see \cite{DKanSR,DKanR}).

The purpose of this paper is to define a notion of homotopy action in Bredon 
equivariant homotopy theory, and describe an associated inductive procedure 
for realizing such an action by a continuous one. 

One might be tempted to say that a homotopy action of $G$ should simply be a 
homotopy-commutative diagram \w[.]{\uX':\OG\op\to\ho\TT} We then have available 
the obstruction theory of Dwyer, Kan, and Smith for rectifying general 
homotopy-commutative diagrams (cf.\ \cite{DKSmH,DKanEH}), which we can use 
to try to lift \w{\uX'} to a strict diagram \w[,]{\uX:\OG\op\to\TT} yielding a 
$G$-space, unique up to Bredon equivalence.

However, the orbit category \w{\OG} can be quite complicated: it includes various 
isomorphisms \w{G/H\cong G/H\sp{a}} for \w[,]{a\in G} and in particular an action of 
\w{N\sb{G}H} as the automorphisms of \w{G/H} for each \w[.]{H\leq G} In the 
Dwyer-Kan-Smith approach, all the morphisms of \w{\OG} are treated on an 
equal footing, and must all be made to fit together at one time 
(with increasing levels of coherence). In particular, it does not allow us 
to interpret the initial data in terms of homotopy actions
of each \w{N\sb{G}H} on \w[.]{\uX'(G/H)} 

The version of homotopy action that we define here involves an ordinary 
diagram of spaces (with no group actions), which we assume for simplicity 
to be strict. We do require a certain amount of equivariant rectification 
in addition, but we keep this to the minimum, and in a form that reduces 
to an elementary lifting problem, in the spirit of \cite{CookR} and \cite{DKanR}, 
starting with certain ordinary homotopy actions of \w[.]{N\sb{G}H}

\begin{mysubsection}{Bredon homotopy actions}\label{sbhaction}
We let $\Lambda$ denote the partially ordered set of subgroups of $G$,
and define a \emph{Bredon homotopy action} of $G$ to consist of:
\begin{itemize}
\item A diagram \w[;]{\wtX:\LG{}\to\TT} 
\item For each conjugacy class \w[,]{\lra{H}} a pointed homotopy action of 
\w{\WH:=N\sb{G}H/H} on the homotopy cofiber \w{\tXHH} of the obvious map 
\w{\hocolim\sb{K>H}\,\wtX(K)\to\wtX(H)} for some representative \w[.]{H\in\lra{H}} 
\end{itemize}

If \w{H'} and $H$ are conjugate in $G$, we must have a homotopy-commuting square:
\mydiagram[\label{eqhosquare}]{
\hocolim\sb{K>H}\,\wtX(K)\ar[rr] \ar[d]\sp{\simeq} && \wtX(H) \ar[d]\sp{\simeq} \\
\hocolim\sb{K'>H'}\,\wtX(K')\ar[rr]&& \wtX(H')~
}
\noindent with vertical homotopy equivalences.
\end{mysubsection}

\begin{mysubsection}{Realizing Bredon homotopy actions}\label{srbhaction}
We wish to realize such a Bredon homotopy action by a topological action,
using descending induction on the subgroups of $G$:
without specifying the $G$-space $\bX$ itself, assume that for some
\w{H\leq G} we have constructed a partial diagram $\uX$ consisting of spaces 
\w{\uX(K)\simeq\wtX(K)} (to be thought of as of ``fixed point sets''
\w{\bX\sp{K}} for the putative $G$-space $\bX$) for all groups
\w[,]{H<K\leq G} together with inclusions \w{i\sp{\ast}:\uX(L)\hra\uX(K)} for 
\w[,]{i:K\hra L} compatible with an action of $G$ on $\uX$ by
\w[.]{\uX(K)\mapsto\uX(K\sp{a})}

Note that we may filter the collection of subgroups of $G$ (or the objects of 
\w[)]{\OG} by letting \w{\F\sb{k}} consist of those subgroups $H$ for which 
there is a chain of proper inclusions \w[.]{H=H\sb{0}<H\sb{1}<\dotsc<H\sb{k}=G} 
If we set \w[,]{\ubX{H}:=\bigcup\sb{H<K}\,\uX(K)} by induction on this filtration
we assume that we have actions of \w{\WH} on \w{\ubX{H}} and \w{\uXHH} (the 
latter realizing the given pointed homotopy action on \w{\tXHH} \wwh see Appendix). 
These fit into a homotopy cofibration sequence:
\begin{myeq}[\label{eqint}]
\ubX{H}~\to~\wtX(H)~\to~\uXHH~.
\end{myeq}
\noindent The key ingredient in the inductive procedure for realizing a 
Bredon homotopy action as above is the ``interpolation'' problem: given two 
$\WH$-spaces such as \w{\uX\sb{H}} and \w[,]{\uXHH} and two maps as in 
\wref[,]{eqint} how to obtain a compatible \ww{\WH}-action on the middle 
space \w[.]{\wtX(H)} This can be reduced to a lifting problem (see 
Propositions \ref{pinterp} and \ref{pintrp}). If we succeed in solving it, 
we have extended our diagram $\uX$ to $H$, too. 

Our main result shows that if this procedure can be completed
for all \w[,]{H\leq G} we obtain a full \ww{\OG\op}-diagram $\uX$, and thus 
a $G$-space $\bX$ realizing the given Bredon homotopy action 
(cf.\ \cite{ElmS}):

\begin{thma}
A Bredon homotopy action \w{\cA:=\lra{\wtX,(\Phi\sp{\ast}\sb{H})\sb{H\leq G}}} 
for a finite group $G$ can be realized by a $G$-space $\bX$ if and only 
if one can inductively construct a sequence of cofibrant diagrams 
\w{(\uX\sb{k}:\F\sb{k}\to\TT)} realizing $\cA$.
Moreover, one can extend \w{\uX\sb{k}} to \w{\uX\sb{k+1}} if and only if for each 
\w[,]{H\subseteq\F\sb{k+1}\setminus\F\sb{k}} one can find a map \w{\Psi\sb{H}} 
making the following diagram of topological spaces commute:
$$
\xymatrix{
&&&&& \bB\GX{\ubX{H}}\\
\bB\WH \ar[rrrrru]\sp{\bB\zeta\sb{\ubX{H}}} \ar[rrrrrd]\sb{\bB\zeta\sb{\uXHH}} 
\ar@{.>}[rrr]^<<<<<<<<<<<<<<<<<<{\Psi\sb{H}} &&& \bB\QXY{j\sb{H}}{q\sb{H}} 
\ar[rru]\sb{\bB\delta\sb{j\sb{H}}\circ\bB\mu} 
\ar[rrd]\sp{\bB\vre\sb{q\sb{H}}\circ\bB\nu} \\
&&&&& \bB\GX{\uXHH}
}
$$
\end{thma}
\noindent [See Theorem \ref{trha} below; the topological monoid \w{\QXY{f}{g}}  
is defined in \S \ref{dspinterpol}].
\end{mysubsection}

\begin{mysubsection}{Related lifting problems}\label{srlp}
Along the way we discuss three related but simpler questions, of independent 
interest, and show how they too may be reduced to lifting problems for 
appropriate fibrations:

\begin{enumerate}
\renewcommand{\labelenumi}{(\roman{enumi})\ }
\item How to extend a $G$-action on a space $\bX$ along a map \w[;]{f:\bX\to\bY}
\item How to lift a $G$-action on $\bY$ along a map \w[;]{f:\bX\to\bY}
\item How to make a map \w{f:\bX\to\bY} between two $G$-spaces into a $G$-map.
\end{enumerate}
See Propositions \ref{prelext} and \ref{prlex}.
\end{mysubsection}

\begin{mysubsection}{Obstructions}\label{sobst}
All these lifting problems have associated obstruction theories, described in 
terms of (Moore-)Postnikov towers (see Proposition \ref{pinterpob}), and thus 
similar in spirit to Cooke's original approach to the realization problem for 
homotopy actions (see \cite{CookR}). These differ from the obstruction theory 
of \cite{DKSmH}, though unfortunately neither version is easily computable.

We may thus conclude from Theorem A that the obstructions of 
Proposition \ref{pinterpob} are the only ones to realizing a Bredon homotopy 
action, and the difference obstructions distinguish between the resulting 
realizations up to $G$-homotopy equivalence (see Corollary \ref{crha}).
\end{mysubsection}

\begin{remark}\label{rha}
The question of realizing homotopy actions is an old one, going back to work 
of Cooke in \cite{CookR} (see also \cite{LSmiRH,OprL,ZabGC,SVogC}). Many 
approaches to this and related problems appear in the literature: since 
(homotopy) actions induce maps between classifying spaces of groups and monoids, 
any information about the latter is relevant to the question at hand. Methods 
for analyzing maps between classifying spaces were developed by Dwyer, 
Zabrodsky, Jackowski, McClure, Oliver, and others in the 1980's 
(cf.\ \cite{DMolH,DZabM,JMOlH}), and later by Grodal and Smith for actions 
on spheres, in \cite{GSmitC}, based on Lannes theory (cf.\ \cite{LannE}). 
Our approach here is more elementary, and perhaps more conceptual, although 
the machinery for calculating our obstructions is not as well-developed.
\end{remark}

\begin{notation}\label{snot}
The category of topological spaces will be denoted by $\TT$, and its
objects will be denoted by boldface letters: $\bX$, \w[.]{\bY\dotsc}
The category of pointed topological spaces \w{\Xs=(\bX,x\sb{0})} is denoted 
by \w[.]{\Ta}

A \emph{$G$-space} is a topological space $\bX$ equipped with a left $G$-action, 
and the category of $G$-spaces with \emph{$G$-maps} (i.e., $G$-equivariant 
continuous maps) will be denoted by \w[.]{\GT} We write \w{\bX\sp{H}} for the 
\emph{fixed point set} \w{\{x\in\bX~:hx=x \ \forall h\in H\}} of $X$ under 
a subgroup \w[.]{H\leq G}

An important example is a $G$-\emph{CW complex}, obtained by attaching 
\emph{$G$-cells} of the  form \w{G/H\times\bD{n+1}} for \w{n\geq -1} 
(see \cite{IllmE1}). For finite $G$, this is equivalent to $\bX$ being a 
CW-complex on which $G$ acts cellularly (see \cite[II, \S 1]{DiecTG}).

An action of a (discrete) group $G$ on $\bX$ is given by a homomorphism
\w[,]{\varphi\sb{\bX}:G\to\AX{\bX}} which we call the \emph{action map} of $\bX$. 
We call the composite \w{\zeta\sb{\bX}:=i\sb{\bX}\circ\varphi\sb{\bX}:G\to\GX{\bX}} the
\emph{monoid action map} of $\bX$, where \w{i\sb{\bX}:\AX{\bX}\hra\GX{\bX}} is the 
inclusion. Similarly, a pointed action of $G$ on \w[,]{\Xs} given by the 
\emph{pointed action map}  \w[,]{\varphi\sb{\Xs}\sp{\ast}:G\to\AXp{\Xs}} has a 
\emph{pointed monoid action map} 
\w[.]{\zeta\sp{\ast}\sb{\bX}:=i\sb{\Xs}\circ\varphi\sp{\ast}\sb{\Xs}:G\to\GXp{\Xs}}
\end{notation}

\begin{mysubsection}{Organization}
\label{sorg}
In Section \ref{cgsoc} we provide some basic background on
$G$-spaces, the orbit category, and equivariant homotopy theory.
In Section \ref{ctga} we address the question of transferring group actions 
along a map (cf.\ \S \ref{srlp}), as preparation for the interpolation problem, 
discussed in Section \ref{ciga} (both for arbitrary groups).
In Section \ref{cdha} we define the notion of a Bredon homotopy action and 
prove our main result (for finite $G$). In the Appendix, we review the notion 
of a  pointed homotopy action.
\end{mysubsection}

%
%

%
%
\sect{$G$-Spaces and the Orbit Category}
\label{cgsoc}

In this section we recall some basic facts about $G$ spaces, and the Borel 
and Bredon approaches to equivariant homotopy theory.

\begin{mysubsection}{Homotopy actions}\label{shaction}
Let \w{\GX{\bX}} denote the strictly associative topological monoid of self-homotopy equivalences of a topological space $\bX$. If $G$ is a group, any monoid map \w{\zeta\sb{\bX}:G\to\GX{\bX}} factors through the submonoid \w{\AX{\bX}} of invertible elements (self-homeomorphisms) in \w[,]{\GX{\bX}} so it makes $\bX$ into a $G$-\emph{space}, equipped with a continuous $G$-\emph{action}.

However, \w{\AX{\bX}} is not a homotopy invariant of $\bX$, while \w{\GX{\bX}} is. A \emph{homotopy action} of $G$ on $\bX$ is therefore defined to be 
the homotopy class of a map \w{\Phi:\BG\to\bB\GX{\bX}} (see \cite{DDKaE,DWilH} 
and compare \cite{SulG}). In particular, a group action determines a homotopy action, by setting \w[.]{\Phi:=B\zeta\sb{\bX}}

If \w{\Xs=(\bX,x\sb{0})} is pointed, \w{\GX{\bX}} has a sub-monoid \w{\GXp{\Xs}} consisting of the pointed self-homotopy equivalences of $\bX$, and a \emph{pointed homotopy action} of $G$ on \w{\Xs} is (the homotopy class of) a map \w[.]{\Phi\sp{\ast}:\BG\to\bB\GXp{\Xs}}

The inclusion \w{j:\GXp{\Xs}\hra\GX{\bX}} fits into a homotopy fibration sequence:
\begin{myeq}[\label{eqhfibseq}]
\GXp{\Xs}~\xra{j}~\GX{\bX}~\xra{\eval\sb{x\sb{0}}}~\bX~\xra{k}~
\bB\GXp{\Xs}~\xra{\bB j}~\bB\GX{\bX}~,
\end{myeq}
\noindent where \w{\bB j} is universal for Hurewicz fibrations with
homotopy fiber $\bX$ (cf.\ \cite{AllaC,StasC}, and see
\cite[Theorem 5.6]{BGMoS} \& \cite[Proposition 4.1]{DZabU}). 

Note that a \emph{free} $G$-space $\bX$ is the total space of a principal 
$G$-bundle over the orbit space \w[,]{\bX/G} which is classified by a map 
\w[.]{\vartheta:\bX/G\to\BG} If we let \w{\Xp~\hra~E\sb{\theta}~\xra{\theta}~\bB G} 
denote the pullback of \w{\bB j} along $\Phi$, this fits into a commuting diagram 
of fibration sequences:
\mydiagram[\label{eqrowcol}]{
\ar @{} [ddrrrr] |<<<<<<<<<<<<<<<<<<<<<<<<<<<<<{\framebox{\scriptsize{PB}}}
\ast \ar[rr] \ar[d] && G \ar[rr]\sp{=} \ar[d] && G \ar[d] \\
\Xp \ar@{^{(}->}[rr]\sp{\simeq} \ar[d]\sp{=} &&
\Xpp \ar@{->>}[rr] \ar@{->>}[d]\sb{\xi} &&
\bE G \ar@{->>}[d] \\
\Xp\ar@{^{(}->}[rr] && E\sb{\theta} \ar@{->>}[rr]\sp{\theta} && \bB G,
}
\noindent thus yielding a (free) topological $G$-action on \w{\Xpp\sim\bX}
(with \w{\theta\in[\Xpp/G,\,\bB G]} corresponding to \w{\vartheta\in[\bX/G,\,\bB G]} 
under this equivalence, if $\bX$ is a free $G$-space).
Since we cannot guarantee that $G$ will act on $\bX$ itself, this is sometimes 
refereed to as a \emph{proxy} action (cf.\ \cite{DWilH}).

Note that for every $G$-space $\bX$ there is a $G$-map \w{\bX\times\EG\to\bX} 
which is a homotopy equivalence (out of a free $G$-space). With this notion of 
$G$-weak equivalence, we obtain the Borel version of equivariant homotopy theory, 
which thus reduces to the study of principal $G$-bundles.

We say that the homotopy action $\Phi$ is \emph{realized} by a free $G$-space 
\w{\Xp} if the corresponding principal $G$-bundle \wref{eqrowcol} is
classified by a map $\theta$ which is the pullback of \wref{eqhfibseq} along 
$\Phi$; we have just seen that any homotopy action is realizable. Similarly, 
any pointed homotopy action is realizable by a pointed topological action 
(see Appendix).
\end{mysubsection}

Let $G$ be a fixed group.  Bredon's approach to $G$-equivariant homotopy theory 
(cf.\ \cite{BredE,ElmS}) reduces the study of a $G$-space $\bX$ to the system 
of fixed point sets under the subgroups of $G$. To describe it, we recall 
the following:

\begin{mysubsection}{The orbit category}\label{socat}
The \emph{orbit category} \w{\OG} of $G$ has the cosets \w{G/H} 
(for each \w[)]{H\leq G} as objects, and $G$-equivariant maps as morphisms.

Any map \w{G/H\to G/K} in \w{\OG} can be factored as
\w{i\sb{\ast}:G/H\to G/K\sp{a\sp{-1}}} (induced by the inclusion
\w[),]{i:H\hra K\sp{a\sp{-1}}} followed by an isomorphism
\w[,]{\cj{K\sp{a\sp{-1}}}{a}:G/K\sp{a\sp{-1}}\to G/K} where
\w[,]{K\sp{a\sp{-1}}:=aKa\sp{-1}} for \w[,]{a\in G} and \w{\cj{K\sp{a\sp{-1}}}{a}} is
induced by the right translation \w[.]{g\mapsto ag}
Two maps \w{\cj{K\sp{a\sp{-1}}}{a}\circ i\sb{\ast}}
and \w{\cj{K\sp{b\sp{-1}}}{b}\circ j\sb{\ast}} from \w{G/H} to \w{G/K}
are the same in \w{\OG} if and only if \w[.]{a\sp{-1}b\in K}
Thus the automorphism group \w{\WH:=\Aut\sb{\OG}(G/H)} of \w{G/H\in\OG} is
\w{N\sb{G}H/H} (where \w{N\sb{G}H} is the normalizer of $H$ in $G$).
\end{mysubsection}

\begin{mysubsection}{$\OG$-diagrams}\label{sgdiagrams}
An \ww{\OG\op}-\emph{diagram} in $\TT$ is a functor \w[,]{\Psi:\OG\op\to\TT} and 
the category of all such will be denoted by \w[.]{\TT\sp{\OG\op}}
The main example we have in mind is the fixed point set diagram $\uX$
associated a $G$-space $\bX$, defined \w[.]{\uX(G/H):=\bX\sp{H}}

Since $\TT$ is a simplicial model category, 
\w{\TT\sp{\OG\op}} has a projective simplicial model category
structure in which a map \w{f:\Psi\to\Psi'} of \ww{\OG\op}-diagrams is a
weak equivalence (respectively, a fibration) if for each \w[,]{H\leq G}
\w{f(G/H):\Psi(G/H)\to\Psi'(G/H)} is a weak equivalence (respectively, a
fibration). See \cite[Theorem 11.7.3]{PHirM}.

There is an analogous simplicial model category structure on \w[,]{\GT}
in which a  $G$-map \w{f:\bX\to\bY} is a weak equivalence (respectively, fibration) 
if for each \w[,]{H\leq G} the map \w{f\rest{\bX\sp{H}}} is a weak equivalence 
(respectively, fibration). See \cite{DKanSR} and compare \cite{PiacH}.

The following result of Elmendorf explains the central role of fixed-point sets 
in Bredon equivariant homotopy theory:
\end{mysubsection}

\begin{thm}[\protect{\cite[Theorem 1]{ElmS}}]\label{telm}
The fixed point set functor sending a $G$-space $\bX$ to the diagram 
\w{\uX:\OG\op\to\TT} has a right adjoint \w[.]{C:\TOG\to\GT}
\end{thm}

\begin{remark}\label{rgcw}
In fact, this adjoint pair constitutes a simplicial Quillen equivalence
between \w{\GT} and \w[.]{\TOG} Moreover, for any $G$-space $\bX$, Elmendorf 
shows that \w{C\uX} is a $G$-CW complex. We therefore may (and shall) assume 
from now on that all our $G$-spaces are $G$-CW complexes.
\end{remark}

\begin{defn}\label{dborel}
A $G$-map \w{h:\bX\to\bY} which at the same time is a (non-equivariant)
homotopy equivalence will be called a \emph{Borel $G$-equivalence}. If
\w{x\sb{0}\in\bX} and \w{y\sb{0}=h(x\sb{0})\in\bY} are $G$-base-points (fixed 
under the $G$ action) and $h$ is a pointed homotopy equivalence, it will be called 
a \emph{pointed Borel $G$-equivalence}.
\end{defn}

\begin{lemma}\label{lgxequiv}
For any homotopy equivalence \w{h:\bX\to\bY} between CW complexes, there is a 
CW complex $\bZ$ with homotopy equivalences \w{i:\bX\to\bZ} and \w{i':\bY\to\bZ} 
such that \w[,]{i\sim h\circ i'} inducing strictly multiplicative monic homotopy 
equivalences \w{i\sb{\star}:\GX{\bX}\to\GX{\bZ}} and 
\w[.]{i'\sb{\star}:\GX{\bY}\to\GX{\bZ}}
\end{lemma}

\begin{proof}
Factoring $h$ as \w{p'\circ i=h} with $i$ a cofibration and \w{p'} a fibration, 
and using the cofibrancy of $\bX$ and $\bY$, we obtain a diagram of 
homotopy equivalences
\mydiagram[\label{eqtriangle}]{
&& \bZ \ar[rrd]\sb{p'} \ar@/_{1.5pc}/[lld]\sb{p}&&\\
\bX \ar[rru]\sb{i} \ar[rrrr]\sp{h} &&&&
\bY \ar@/_{1.5pc}/[llu]\sb{i'}
}
\noindent with \w{p\circ i=\Id\sb{\bX}} and \w[.]{p'\circ i'=\Id\sb{\bY}}
Define \w{i\sb{\star}:\GX{\bX}\to\GX{\bZ}} by
\w{\varphi\mapsto i\circ\varphi\circ p} (for any homotopy
equivalence \w[).]{\varphi:\bX\to\bX} Because
\w[,]{p\circ i=\Id\sb{\bX}} the map \w{i\sb{\star}} is monic, preserves
compositions, and has a (non-monoidal) homotopy inverse
\w[.]{p\sb{\ast}:\GX{\bZ}\to\GX{\bX}} Similarly for \w[.]{i'}
\end{proof}

\begin{remark}\label{rgxequiv}
Any homotopy equivalence \w{h:\bX\to\bY} induces a homotopy equivalence
\w{\bB\GX{\bX}\simeq\bB\GX{\bY}} (cf.\ \cite[Satz 7.7]{MFucH}).
In fact, we can apply the classifying space functor $\bB$ to the maps 
\w{i\sb{\star}} and \w[,]{i'\sb{\star}} obtaining homotopy equivalences:
$$
\bB\GX{\bX}~\xra{\bB i\sb{\star}}~\bB\GX{\bZ}~
\xra{(\bB i'\sb{\star})\sp{-1}}~\bB\GX{\bY}~,
$$
\noindent whose composite is denoted by \w{\bB h\sb{\ast}} (well-defined up to 
homotopy). Similarly in the pointed case.
\end{remark}

%
%
\sect{Transferring group actions}
\label{ctga}

In this and the following section $G$ can be any topological group. Given a 
map \w[,]{f:\bX\to\bY} consider the questions of:

\begin{enumerate}
\renewcommand{\labelenumi}{$\bullet$ \ }
\item Transferring a given $G$-action on $\bX$ along $f$ to $\bY$. or conversely.
\item Making $f$ equivariant with respect to given actions on
both $\bX$ and $\bY$.
\end{enumerate}

In the spirit of \cite{DKanR}, we shall show how they can be reduced to suitable 
lifting problems. First, we make the questions more precise:

\begin{defn}\label{drelext}
Given any map \w[,]{f:\bX\to\bY} a $G$-map \w{f':\bX'\to\bY'} is:
\begin{enumerate}
\renewcommand{\labelenumi}{(\roman{enumi})\ }
\item a \emph{right transfer} of a $G$-action on $\bX$
\emph{along $f$} if we have a homotopy-commutative diagram
\mydiagram[\label{eqrtransf}]{
\bX \ar[rr]\sp{f} && \bY \ar[d]\sb{\simeq}\sp{h}\\
\bX' \ar[u]\sb{\simeq}\sp{k} \ar[rr]\sp{f'} && \bY'
}
\noindent in which $h$ is a  homotopy equivalence,
and $k$ is a Borel $G$-equivalence.
\item a \emph{left transfer} of a $G$-action on $\bY$
\emph{along $f$} if we have a diagram
\mydiagram[\label{eqltrans}]{
\bX \ar[rr]\sp{f} && \bY &&
\bY'' \ar[ll]\sp{\simeq}\sb{m} \ar[lld]\sb{\simeq}\sp{n} \\
\bX' \ar[u]\sb{\simeq}\sp{k}  \ar[rr]\sp{f'} && \bY' &&
}
\noindent in which $k$ is a  homotopy equivalence,
$m$ and $n$ are Borel $G$-equivalences, which becomes homotopy-commutative 
after inverting $m$ or $n$ (up to homotopy).
\item a \emph{compatible $G$-map} for $f$ \emph{with respect to $G$-actions} 
on $\bX$ and $\bY$ if we have a diagram
\mydiagram[\label{eqcompat}]{
\bX \ar[rr]\sp{f} && \bY &&
\bY'' \ar[ll]\sp{\simeq}\sb{m} \ar[lld]\sb{\simeq}\sp{n} \\
\bX' \ar[u]\sb{\simeq}\sp{k} \ar[rr]\sp{f'} && \bY' &&
}
\noindent in which $k$, $m$,  and $n$ are all Borel $G$-equivalences, which 
becomes homotopy-commutative after inverting $m$ or $n$.
\end{enumerate}
\end{defn}

In order to describe the conditions under which such transfers exist, we 
require also the following construction:

\begin{defn}\label{dspext}
For any map \w[,]{f:\bX\to\bY} let \w{\PXY{f}} denote the homotopy pullback:
\mydiagram[\label{eqspext}]{
\ar @{} [drr] |<<<<<{\framebox{\scriptsize{PB}}}
\PXY{f} \ar@{->>}[d]\sp{\delta} \ar[rr]\sp{\vre} &&
\GX{\bY} \ar@{->>}[d]\sp{f\sp{\ast}}\\
\GX{\bX} \ar[rr]\sb{f\sb{\ast}} && \map(\bX,\bY)~.
}
\noindent This can be constructed explicitly in two ways: if we change $f$ 
into a cofibration, the map \w{f\sp{\ast}:\map(\bY,\bY)\to\map(\bX,\bY)} is 
a fibration, so its restriction to \w{\GX{\bY}} is a fibration, too (since the 
latter is just a union of path components of \w[).]{\map(\bY,\bY)} In this case, 
the strict pullback is actually the homotopy pullback. Similarly when $f$ is 
a fibration, so \w{f\sb{\ast}} is a fibration.

Using such a strict model, we see that \w{\PXY{f}} is a sub-monoid of the 
strictly associative monoid \w[.]{\GX{\bX}\times\GX{\bY}} Moreover, it is 
grouplike, since \w{(g,h)\in\PXY{f}} means that \w{f\circ g=h\circ f} (for
\w{g\in\GX{\bX}} and \w[),]{h\in\GX{\bY}} and thus
\w[.]{f\circ g\sp{-1}\sim h\sp{-1}\circ f} If $f$ is either a fibration or a 
cofibration, we can use \cite[Lemma 4.16]{BJTurR} to change \w{h\sp{-1}} 
(respectively,  \w[)]{g\sp{-1}} up to homotopy to get 
\w{f\circ g\sp{-1}=h\sp{-1}\circ f} and thus \w[,]{(g\sp{-1},h\sp{-1})\in\PXY{f}} 
too. The maps $\delta$ and $\vre$ (the restrictions of the
structure maps for \w[)]{\tPXY{f}} are monoid maps. Evidently \w{\PXY{f}} is 
a homotopy invariant of $f$.
\end{defn}

With these notions we then have the following:

\begin{prop}\label{prelext}
Let \w{f:\bX\to\bY} be any map in $\TT$.

\begin{enumerate}
\renewcommand{\labelenumi}{(\roman{enumi})\ }
\item There is a right transfer of a $G$-action on $\bX$ (with monoid action 
map \w[)]{\zeta\sb{\bX}:G\to\GX{\bX}} along $f$ if and only if there is a map 
$\Psi$ making the following diagram commute up to homotopy:
\mydiagram[\label{eqrighttr}]{
&&\bB\PXY{f} \ar[d]\sp{\bB\delta} \\
\BG \ar[rru]\sp{\Psi} \ar[rr]\sb{\bB\zeta\sb{\bX}} && \bB\GX{\bX} .
}
\item There is a left transfer of a $G$-action on $\bY$ (with monoid action map 
\w[)]{\zeta\sb{\bY}:G\to\GX{\bY}} along $f$ if and only if there is a map $\Psi$ 
making the following diagram commute up to homotopy:
\mydiagram[\label{eqlefttr}]{
&&\bB\PXY{f} \ar[d]\sp{\bB\vre} \\
\BG \ar[rru]\sp{\Psi} \ar[rr]\sb{\bB\zeta\sb{\bY}} && \bB\GX{\bY} .
}
\item There is a compatible $G$-map for $f$ with respect to
$G$-actions on $\bX$ and $\bY$ (with monoid action maps 
\w{\zeta\sb{\bX}:G\to\GX{\bX}}  and \w[)]{\zeta\sb{\bY}:G\to\GX{\bY}} if and only 
if there is a map $\Psi$ making the following diagram commute up to homotopy:
\mydiagram[\label{eqcompatible}]{
&&& \bB\PXY{f} \ar[d]\sp{(\bB\delta,\bB\vre)} \\
\BG \ar[rrru]\sp{\Psi} \ar[rrr]\sb{(\bB\zeta\sb{\bX},\bB\zeta\sb{\bY})} &&& 
\bB\GX{\bX}\times\bB\GX{\bY} .
}
\end{enumerate}
\end{prop}

\begin{proof}
\noindent (i)\ If $\bX$ is a $G$-space, and we have a right transfer
\w[,]{f':\bX'\to\bY'}  of the $G$-action along $f$, we may change \w{f'} into a
$G$-cofibration, and the monoid action maps
\w{\zeta\sb{\bX'}:G\to\GX{\bX'}} and \w{\zeta\sb{\bY'}:G\to\GX{\bY'}} then fit
together to define a monoid map \w{z:G\to\tPXY{f'}} in
\wref[,]{eqspext} which actually lands in \w[,]{\PXY{f'}}
since $G$ is a group. Because \w{\bB\zeta\sb{\bX}} is just
\w[,]{\bB\zeta\sb{\bX'}} up to homotopy, \w{\bB z:\BG\to\bB\PXY{f'}} is the required
lift in \wref[.]{eqrighttr}

Conversely, given a lift $\Psi$ in \wref[,]{eqrighttr} by applying Kan's 
$G$-functor to \wref[,]{eqrighttr} realizing, and then taking cofibrant replacement 
in the model category of strictly associative topological monoids (see 
\cite[Theorem B]{SVogCM}), we obtain a diagram of cofibrant (and grouplike) 
topological monoids:
\mydiagram[\label{eqstrictgp}]{
&& \hPXY{f} \ar[d]\sp{\hdel} \ar[rrd]\sp{\hvr} &&\\
\hG  \ar[rru]\sp{\hrho} && \hGX{\bX} && \hGX{\bY}~,
}
\noindent with \w{\hPXY{f}} weakly equivalent to \w[.]{\PXY{f}}

Since \w[,]{\bB\GX{\hX}\simeq\bB\GX{\bX}} by pulling back \wref{eqhfibseq} we 
obtain a monoid action of $\hG$ on \w[.]{\hX\simeq\bX}
By \cite[Theorem 5.8]{PrezH} (see also \cite{DLashP,MFucD}, 
\cite[\S 7,9]{MayCF},
\cite[\S 5]{BootEH} and \cite{DDKaE}), this is classified by a map
\w[.]{\hat{\theta}:E\sb{\hat{\theta}}\to\bB\hG} Up to homotopy, $\hat{\theta}$ 
corresponds to the map $\theta$ in the fibre bundle sequence:
\begin{myeq}[\label{eqclmap}]
\bX':~=\EG\times\bX~\xra{}~E\sb{\theta}~\xra{\theta}~\BG
\end{myeq}
\noindent classifying the free $G$-action on \w{\bX'} (Borel equivalent to the 
given $\bX$).  Similarly, we get a free $\hG$-action on \w[,]{\hY\simeq\bY} 
classified by \w{\hat{\kappa}:E\sb{\hat{\kappa}}\to\bB\hG} Moreover, we have a map
\w{\hf:\hX\to\hY} (which is just \w[,]{f:X\to Y} up to homotopy), and
we may assume that $\hf$ is itself a cofibration (for example, by carrying out 
the above construction in simplicial sets, and replacing $\hY$ by \w{\hY\times\bC\hX}
before realizing).

As in \wref[,]{eqspext} we obtain a commuting diagram
\mydiagram[\label{eqhsquare}]{
\PXY{\hf} \ar[d]\sb{\hdel'} \ar[rr]\sp{\hvr'} &&
\GX{\hY}  \ar@{->>}[d]\sp{\hf\sp{\ast}}\\
\GX{\hX} \ar[rr]\sb{\hf\sb{\ast}} && \map(\hX,\hY)~,
}
\noindent in which \w{\hf\sp{\ast}} is a fibration, so \w{\hdel'} is,
too.

Because \w[,]{\hX\simeq\bX} \w{\hGX{\bX}} and \w{\GX{\bX}} are weakly equivalent, 
and since the former is cofibrant and the latter is fibrant, we have a weak 
equivalence of monoids \w[,]{k:\hGX{\bX}\to\GX{\bX}} and similarly 
\w[.]{\ell:\hGX{\bY}\simeq\GX{\bY}} Moreover, since \wref{eqhsquare}
is a homotopy pullback, \w{\PXY{f}} and \w{\PXY{\hf}} are weakly equivalent, 
and again we have a weak equivalence of monoids 
\w[.]{h:\hPXY{f}\xra{\simeq}\PXY{\hf}} Thus the strict diagram \wref{eqhsquare} 
fits into a homotopy commutative diagram:
\mydiagram[\label{eqhext}]{
\ar @{} [ddrr] |<<<<<<<<<<<<<<<<<<{\framebox{\scriptsize{A}}}
\ar @{} [drrrr] |<<<<<<<<<<<<<<<<<<<<<{\framebox{\scriptsize{B}}}
\hPXY{f}\ \ar[d]\sb{\hdel} \ar[rrd]\sp{h} \ar[rr]\sp{\hvr} &&
\hGX{\bY} \ar[rrd]\sp{\ell} && \\
\hGX{\bX} \ar[rrd]\sb{k} &&
\PXY{\hf}\ \ar[d]\sp{\hdel'} \ar[rr]\sp{\hvr'} &&
\GX{\hY}  \ar@{->>}[d]\sp{\hf\sp{\ast}}\\
&&\GX{\hX} \ar[rr]\sb{\hf\sb{\ast}} && \map(\hX,\hY)~.
}
\noindent In the model category of strictly associative monoids, we can replace 
$h$ by another weak equivalence of monoids  making \w{\framebox{\scriptsize{A}}}
commute on the nose (cf.\ \cite[Lemma 4.16]{BJTurR}), and then changing $\ell$ 
into a fibration, we may replace $\hvr$ by a map making 
\w{\framebox{\scriptsize{B}}} commute strictly, too, without changing \w[.]{\hPXY{f}}

Composing the monoid map \w{\hrho:\hG\to\hPXY{f}} of \wref{eqstrictgp} with 
\w{k\circ\hdel:\hPXY{f}\to\GX{\hX}}
and \w[,]{\ell\circ\hvr:\hPXY{f}\to\GX{\hY}} we obtain monoid
action maps \w{\hz\sb{\hX}:\hG\to\GX{\hX}} and
\w{\hz\sb{\hY}:\hG\to\GX{\hY}} making $\hX$ and $\hY$ into strict
$\hG$-spaces, with \w{\hf:\hX\hra\hY} a $\hG$-map (which is a
cofibration). It therefore fits into a commuting diagram of principal
$\hG$-bundles
\mydiagram[\label{eqhbundle}]{
\hX\ar[d] \ar@{^{(}->}[rr]\sp{\hf} && \hY \ar[d] \\
E\sb{\hat{\theta}} \ar[rd]\sb{\hat{\theta}} \ar[rr]\sp{\hat{\beta}} &&
E\sb{\hat{\kappa}} \ar[ld]\sp{\hat{\kappa}} \\
 & \bB\hG~. &
}
\noindent Here \w{\hat{\beta}} is obtained by realizing the bar construction 
for the $\hG$-actions on $\hX$ and $\hY$, respectively (see \cite[\S 5]{PrezH}),
so it commutes up to homotopy with the classifying maps $\hat{\theta}$ and 
$\hat{\kappa}$ for the two bundles, where (as noted above) up to homotopy 
$\hat{\theta}$ is just \w[,]{\theta:E\sb{\theta}\to\BG} classifying the free 
$G$-space \w[.]{\bX'}

Let \w{\kappa:E\sb{\hat{\kappa}}\to\BG} denote the composite of \w{\hat{\kappa}} 
with \w[,]{\bB\hG\simeq\BG} classifying a free $G$-bundle
\w{\bY'\to E\sb{\hat{\kappa}}} with \w[.]{\bY'\simeq\hY\simeq\bY}
If we also let \w{\beta:E\sb{\theta}\to E\sb{\hat{\kappa}}} denote the composite 
of \w{\hat{\beta}} with \w[,]{E\sb{\theta}\simeq E\sb{\hat{\theta}}} then 
\w[,]{\kappa\circ\beta\simeq\theta} so we have a map
\mydiagram[\label{eqoldbundle}]{
\bX' \ar[d] \ar[rr]\sp{f'} && \bY' \ar[d] \\
E\sb{\theta} \ar[rd]\sb{\theta} \ar[rr]\sp{\beta} &&
E\sb{\hat{\kappa}} \ar[ld]\sp{\kappa} \\
 & \BG &
}
\noindent of principal $G$-bundles, so in particular, \w{f'} is a $G$-map\vsm.

\noindent Statements (ii) and (iii) are proven analogously.
\end{proof}

Compare \cite[Proposition 2.2]{ZabGC} for the compatibility version for 
homotopy actions.

\begin{prop}\label{prlex}
Let \w{f:\bX\to\bY} be any map. In a right transfer of a $G$-action on $\bX$ 
along $f$, we may assume that $k$ in \wref{eqrtransf} is a homeomorphism; 
in a left transfer of a $G$-action on $\bY$ along $f$, we may assume that 
$m$ and $n$ in \wref{eqltrans} are homeomorphisms; and
in a compatible $G$-map for $f$ with respect to $G$-actions on $\bX$ and
$\bY$, we may assume that either $k$, or $m$ and $n$, are homeomorphisms 
in \wref[.]{eqcompat}
\end{prop}

\begin{proof}
If the action of $G$ on $\bX$ is \emph{free} (and \wref{eqrighttr} holds),
we can replace \w{\bX':=\EG\times\bX} by $\bX$ in \wref[,]{eqclmap} and therefore 
also in \wref[,]{eqoldbundle}
so we have a right transfer \w{f':\bX\hra\bY'} (along \w[)]{f:\bX\to\bY} which 
is a $G$-map.

The same argument shows that if the action of $G$ on $\bY$ is free (and 
\wref{eqlefttr} holds), it has a left transfer along $f$ to a fibration 
\w{f':\bX'\epic\bY} which is a $G$-map. Moreover, given free $G$-actions on 
both $\bX$ and $\bY$, any \w{f:\bX\to\bY} has a compatible $G$-map 
\w{f':\bX\to\bY} with the same source and target (if \wref{eqcompatible} holds).

Since every $G$-space $\bY$ has a Borel $G$-equivalence \w[,]{h:\bY'\to\bY} where 
\w{\bY'} is a free $G$-space, we do not actually need to assume that the action 
on $\bY$ is free, because we can compose the left transfer or compatible map 
\w{f'} with this $h$. Thus we may always assume that $\ell$, $m$, and $n$ are 
homeomorphisms.

Finally, given a $G$-space $\bX$, we may replace it by the free $G$-space
\w{\bX':=\bX\times\EG} and produce a $G$-map \w{f':\bX'\hra\bY'} which is
a cofibration. Then taking the pushout:
\mydiagram[\label{eqgpushout}]{
\ar @{} [drr]|<<<<<<<<<<<<<<{\framebox{\scriptsize{PO}}}
\bX' \ar[d]\sp{q}\sb{\simeq} \ar@{^{(}->}[rr]\sp{f'} && \bY'\ar[d]\sp{r}\sb{\simeq} \\
\bX \ar@{^{(}->}[rr]\sb{f''} && \bY'',
}
\noindent we see that all maps are $G$-maps; \w[,]{f'} and thus \w[,]{f''} are 
cofibrations, so this is a homotopy pushout in $\TT$, and since
$q$ is a homotopy equivalences, so is $r$.
\end{proof}

\begin{mysubsection}{Applications}\label{sapplic}
In general, the lifting problems of Proposition \ref{prelext} are hard to solve. 
However, in certain cases the obstructions to obtain the relevant liftings may 
be computable, or may vanish for dimension reasons. For example:

\begin{enumerate}
\renewcommand{\labelenumi}{(\roman{enumi})\ }
\item When \w{\bX=K(\pi,n)} is an Eilenberg-Mac~Lane space, then:
\begin{myeq}[\label{eqseem}]
\GX{\bX}~\simeq~K(\pi, n)\times\Aut(\pi)
\end{myeq}
\noindent (as a monoid) is a semi-direct product of the Eilenberg-Mac~Lane space 
itself and the discrete group \w[,]{\Aut(\pi)}
while \w{\GXp{X}\simeq \Aut(\pi)\cong\pi\sb{0}\GXp{X}} is homotopically discrete 
(see \cite[Proposition 25.2]{MayS}).

Thus if both $\bX$ and $\bY$ are Eilenberg-Mac~Lane spaces, all but the pullback 
itself in \wref{eqspext} are generalized Eilenberg-Mac~Lane spaces, so each of 
the lifting problems \wref[,]{eqrighttr} \wref[,]{eqlefttr} and \wref{eqcompatible} 
reduces to an algebraic question about certain classes in the cohomology of 
\w{\BG} (as expected).
\item A more interesting example is when each of $\bX$ and $\bY$ has only two 
non-trivial homotopy groups (see \S \ref{eginterp} below). In this case the 
homotopy groups of \w{\GX{\bX}} and \w{\GX{\bY}} are completely known by 
\cite[\S 3]{DidHH} (see also \cite{MolS}), and in particular
if \w{\pi\sb{i}\bX=0} for \w{i\neq k,m} \wb[,]{k<m} then \w{\pi\sb{i}\GX{\bX}=0} 
unless \w[.]{k\leq i\leq m} Using the Postnikov tower for $\bY$ we can also 
determine \w[,]{\pi\sb{\ast}\map(\bX,\bY)} up to extension. Therefore the homotopy 
groups of \w{\PXY{f}} may also be determined, up to extension.

Since we need not assume $G$ is finite, \w{H\sp{i}(\BG;\pi)} may vanish for large
enough $i$, at least when $\pi$ is one of the groups \w{\pi\sb{\ast}\PXY{f}} 
Thus in certain cases we can show that there is no obstruction to solving the 
lifting problems.
\item The case when $\bX$ and $\bY$ are spheres has been the subject of intense 
study over the years, beginning with \cite{PSmiF}.  Moreover, much is known about
\w{\GX{\bS{n}}} and \w{\GXp{\bS{n}}} (see, e.g., \cite{HansHC,HansHG}). Therefore, 
one might be able to compute obstructions to extending or lifting certain group 
actions on spheres along some map \w[,]{f:\bS{n}\to\bS{m}} or making $f$ equivariant.
\end{enumerate}
\end{mysubsection}

%
%
\sect{Interpolating group actions}
\label{ciga}

For our approach to the realization problem for Bredon homotopy actions, we need 
to consider a slightly more complicated situation than that studied in the previous 
section: assume given a sequence of maps
\begin{myeq}[\label{eqinterp}]
\bX~\xra{f}~\bY~\xra{g}~\bZ
\end{myeq}
\noindent with given $G$-actions on $\bX$ and $\bZ$, for which we want to find a 
compatible $G$-action on $\bY$.

\begin{defn}\label{dinterp}
A \emph{$G$-interpolation} for two $G$-spaces $\bX$ and $\bZ$ and maps as in 
\wref{eqinterp} is a pair of $G$-maps \w{f':\bX'\to\bY'} and \w{g':\bY'\to\bZ'} 
fitting into a diagram
\mydiagram[\label{eqinterpol}]{
&& \bX \ar[rr]\sp{f} && \bY \ar[d]\sp{h} \ar[rr]\sp{g} && \bZ &&
\bZ'' \ar[ll]\sp{\simeq}\sb{m} \ar[lld]\sb{\simeq}\sp{n} \\
\bX'' \ar[rru]\sb{\simeq}\sp{k} \ar[rr]\sp{\simeq}\sb{\ell} && 
\bX' \ar[rr]\sp{f'} && \bY' \ar[rr]\sp{g'} && \bZ' &&
}
\noindent in which $k$, $\ell$, $m$,  and $n$ are all homotopy equivalences and 
$G$-maps, and $h$ is a homotopy equivalence, which becomes homotopy-commutative 
after inverting $m$ or $n$ (up to homotopy).
\end{defn}

\begin{defn}\label{dspinterpol}
Given two composable maps as in \wref[,]{eqinterp} let \w{\QXY{f}{g}} denote 
the homotopy pullback in:
\mydiagram[\label{eqspinterpol}]{
\ar @{} [drr] |<<<<{\framebox{\scriptsize{PB}}}
\QXY{f}{g} \ar[d]\sb{\mu\sb{f}} \ar[rr]\sp{\nu\sb{g}} &&
\PXY{g} \ar@{->>}[d]\sp{\delta\sb{g}}\\
\PXY{f} \ar[rr]\sb{\vre\sb{f}} && \GX{\bY}~.
}
\noindent (see  \S \ref{dspext}). If $f$ and $g$ are cofibrations, both 
\w{\PXY{f}} and \w{\PXY{g}} are actually pullbacks, and the map 
\w{\delta:\PXY{g}\to\GX{\bY}} is a fibration, so \w{\QXY{f}{g}} is the ordinary 
pullback. Furthermore, it is a grouplike strictly associative monoid, and the 
maps $\mu$ and $\nu$ are monoid maps.
\end{defn}

\begin{prop}\label{pinterp}
Two maps as in \wref{eqinterp} for $G$-spaces $\bX$ and $\bZ$ (with monoid action 
maps \w{\zeta\sb{\bX}:G\to\GX{\bX}} and \w[,]{\zeta\sb{\bZ}:G\to\GX{\bZ}} 
respectively) have a $G$-interpolation if and only if there is a map $\Psi$ 
making the following diagram commute up to homotopy:
\mydiagram[\label{eqliftinterp}]{
&&&&& \bB\GX{\bX}\\
\bB G \ar[rrrrru]\sp{\bB\zeta\sb{\bX}} \ar[rrrrrd]\sb{\bB\zeta\sb{\bZ}} 
\ar@{.>}[rrr]^<<<<<<<<<<<<<<<<<<{\Psi} &&& \bB\QXY{f}{g} 
\ar[rru]\sb{\bB\delta\sb{f}\circ\bB\mu} \ar[rrd]\sp{\bB\vre\sb{g}\circ\bB\nu} \\
&&&&& \bB\GX{\bZ}.
}
\end{prop}

\begin{proof}
Given a $G$-interpolation, the diagram \wref{eqliftinterp} is obtained by 
applying $\bB$ to the corresponding monoid action maps.

Conversely, a homotopy commutative diagram \wref{eqliftinterp} may be lifted 
(together with \wref[)]{eqspinterpol} to a commuting diagram of topological monoids:

\mydiagram[\label{eqinterptg}]{
&\hG \ar[rrr]\sp{\hrho} &&&
\hQXY{f}{g} \ar[lld]\sp{\widehat{\mu}\sb{f}} 
\ar[rrd]\sb{\widehat{\nu}\sb{g}}  &&&& \\
&& \hPXY{f} \ar[lld]\sp{\hdel\sb{f}} \ar[rrd]\sb{\hvr\sb{f}} &&&&
\hPXY{g} \ar[lld]\sp{\hdel\sb{g}} \ar[rrd]\sb{\hvr\sb{g}} && \\
\hGX{\bX} &&&& \hGX{\bY} &&&& \hGX{\bZ}
}
\noindent and we have spaces \w[,]{\hX\simeq\bX} \w[,]{\hY\simeq\bY} and 
\w{\hZ\simeq\bZ} on which \w[,]{\hGX{\bX}} \w[,]{\hGX{\bY}} and \w[,]{\hGX{\bZ}} 
respectively act. Moreover, we have maps \w{\hf:\hX\to\hY} and \w{\hg:\hY\to\hZ} 
(corresponding up to homotopy to $f$ and $g$, respectively), and as in the proof 
of Proposition \ref{prelext}, we may assume $\hf$ and $\hg$ are cofibrations.

Therefore, \wref{eqinterptg} fits into a diagram:
\mydiagram[\label{eqinterpoltg}]{
&&&&
\hQXY{f}{g} \ar[lld]\sp{\widehat{\mu}\sb{f}} \ar[rrd]\sb{\widehat{\nu}\sb{g}} &&&& \\
&& \hPXY{f} \ar@{.>}[dd]\sp{h} \ar[lld]\sp{\hdel\sb{f}} \ar[rrd]\sb{\hvr\sb{f}} &&&&
\hPXY{g} \ar@{.>}[dd]\sp{k} \ar[lld]\sp{\hdel\sb{g}} \ar[rrd]\sb{\hvr\sb{g}} && \\
\hGX{\bX} \ar[dd]\sp{\ell} &&&& \hGX{\bY} \ar[dd]\sp{m} &&&& 
\hGX{\bZ} \ar[dd]\sp{n}\\
&& \PXY{\hf} \ar[lld]\sp{\delta\sb{\hf}} \ar[rrd]\sb{\vre\sb{\hf}} &&&&
\PXY{\hg} \ar[lld]\sp{\delta\sb{\hg}} \ar[rrd]\sb{\vre\sb{\hg}} && \\
\GX{\hX} &&&& \GX{\hY} &&&& \GX{\hZ}
}
\noindent in which \w{\PXY{\hf}} and \w{\PXY{\hg}} are homotopy limits, so 
the homotopy equivalences $\ell$, $m$, and $n$ induce homotopy equivalences $h$ 
and $k$ making \wref{eqinterpoltg} homotopy-commutative. Since \w{\QXY{\hf}{\hg}} 
is also a homotopy limit, $h$ and $k$ induce a homotopy equivalence 
\w[.]{p:\hQXY{f}{g}\to\QXY{\hf}{\hg}}

Composing \w{p\circ\hrho:\hG\to\QXY{\hf}{\hg}} of diagram \wref{eqinterptg} 
with the appropriate structure maps for \w{\QXY{\hf}{\hg}} yields monoid 
action maps \w[,]{\hz\sb{\hX}:\hG\to\GX{\hX}} \w[,]{\hz\sb{\hY}:\hG\to\GX{\hY}} 
and \w{\hz\sb{\hZ}:\hG\to\GX{\hZ}} making $\hf$ and $\hg$ into $\hG$-equivariant 
maps (by definition of \w[),]{\QXY{\hf}{\hg}} with \w{\hz\sb{\hX}} and 
\w{\hz\sb{\hZ}} corresponding up to homotopy to the given monoid action maps 
\w{\zeta\sb{\bX}:G\to\GX{\bX}} and \w[.]{\zeta\sb{\bZ}:G\to\GX{\bZ}}

Passing to the classifying maps for the corresponding principle $G$- and 
$\hG$-bundles as in the proof of Proposition \ref{prelext}, we obtain the 
required $G$-interpolation.
\end{proof}

We now have the following analogue of Proposition \ref{prlex}:

\begin{prop}\label{pintrp}
If \w{p:=g\circ f:\bX\to\bZ} is a $G$-map in \wref[,]{eqinterp}
in any $G$-interpolation for we may assume that $k$, $\ell$, $m$, and $n$ are 
homeomorphisms in \wref[,]{eqinterpol} and that \w[.]{g'\circ f'=p}
\end{prop}

\begin{proof}
In the proof of Proposition \ref{pinterp} we saw that the lifting $\Psi$ in 
\wref{eqliftinterp} allows us to factor the map of free $G$-spaces (i.e., total 
spaces of principle $G$-bundles)
$$
\bX':=\bX\times\EG~\xra{p':=p\times\Id\sb{\EG}}~\bZ\times\EG=:\bZ'
$$
\noindent as the composite of two maps of free $G$-spaces 
\w{\bX'\xra{f'}\bY'\xra{g'}\bZ'} with \w{\bY\simeq\bY'} (using a 
homotopy-factorization of the corresponding classifying maps of the bundles). 
Applying the (homotopy) pushout \wref{eqgpushout} we obtain a commutative diagram 
of $G$-spaces:
\mydiagram[\label{eqgpshout}]{
\ar @{} [drr]|<<<<<<<<<<<<<<{\framebox{\scriptsize{PO}}}
\bX' \ar[d]\sp{q\sb{\bX}}\sb{\simeq} \ar@{^{(}->}[rr]\sp{f'} && 
\bY' \ar[d]\sp{r}\sb{\simeq} \ar[rr]\sp{g'} && \bZ' \ar[rr]\sp{q\sb{\bZ}} && \bZ\\
\bX \ar@{^{(}->}[rr]\sb{f''} && \bY'' \ar@{.>}[rrrru]\sb{g''} &&&&
}
\noindent where the map \w{g''} out of the pushout is induced by
\w{q\sb{\bZ}\circ g':\bY'\to \bZ} and \w[,]{p:=f\circ g:\bX\to \bZ} which agree 
on \w{\bX'} since \w[.]{p'=g'\circ f'=(g\circ f)\times\Id\sb{\EG}}
\end{proof}

\begin{defn}\label{dinterpob}
Given two maps as in \wref{eqinterp} for $G$-spaces $\bX$ and $\bZ$, let 
\w{\rho:\bB\QXY{f}{g}\to\bB\GX{\bX}\times\bB\GX{\bZ}} be the map 
\w{(\bB\delta\sb{f}\circ\bB\mu,\,\bB\vre\sb{g}\circ\bB\nu)} of
\wref[,]{eqliftinterp} and let:
\mydiagram[\label{eqinterpob}]{
\bB\QXY{f}{g}\ \ar@/^{2.5pc}/[rrrrrr]\sp{\rho}
\dotsc \ar[r] & W\sb{n+1} \ar[rr]\sp{p\sb{n+1}} && 
W\sb{n} \ar[rr]\sp{p\sb{n}} && W\sb{n-1} &
\dotsc \bB\GX{\bX}\times\bB\GX{\bZ}
}
\noindent be the Moore-Postnikov tower for $\rho$,
with \w{W\sb{0}:=\bB\GX{\bX}\times\bB\GX{\bZ}} (cf.\ \cite[VI, 3.9]{GJarS}).

If $F$ denotes the homotopy fiber of $\rho$, then up to homotopy each map
\w{p\sb{n+1}:W\sb{n+1}\to W\sb{n}} is a fibration with fiber 
\w[,]{K(\pi\sb{n+1}F,n+1)}  which is classified by a map
\w{\tilde{k}\sb{n}:W\sb{n}\to K\sb{\pi\sb{1}Q\sb{n}}(\pi\sb{n+1}F,n+2)} (see
\cite[Theorem 3.4]{RobM}). Assume by induction on \w{n\geq 0} that we have 
constructed a lift \w{g\sb{n}:\BG\to W\sb{n}} for
\w[.]{g\sb{0}:=(\bB\zeta\sb{\bX},\bB\zeta\sb{\bZ}):\BG\to W\sb{0}=
\bB\GX{\bX}\times\bB\GX{\bZ}}  
Then the \emph{$n$-th obstruction class} for this lift is
$$
[\tilde{k}\sb{n}\circ g\sb{n}]\in[\BG,\,K\sb{\pi\sb{1}Q\sb{n}}(\pi\sb{n+1}F,n+2)]~
\cong~H\sp{n+2}(G;\,\pi\sb{n+1}F)~,
$$
\noindent where $G$ acts on \w{\pi\sb{n+1}F} via 
\w[.]{(g\sb{n})\sb{\#}:G\to\pi\sb{1}W\sb{n}}
\end{defn}

From Proposition \ref{pinterp} we deduce:

\begin{prop}\label{pinterpob}
Two maps as in \wref{eqinterp} for $G$-spaces $\bX$ and $\bZ$ (with monoid action 
maps \w{\zeta\sb{\bX}:G\to\GX{\bX}} and \w[,]{\zeta\sb{\bZ}:G\to\GX{\bZ}} 
respectively) have a $G$-interpolation if and only if the obstruction classes
\w{[\tilde{k}\sb{n}\circ g\sb{n}]\in H\sp{n+2}(G;\,\pi\sb{n+1}F)} successively 
vanish, for some sequence of lifts. There is also a sequence of 
\emph{difference obstructions} in \w{H\sp{n+1}(G;\,\pi\sb{n+1}F)} for 
distinguishing between non-homotopic lifts in \wref[.]{eqliftinterp}
\end{prop}

\begin{remark}\label{robst}
As with any obstruction theory, non-vanishing of a cohomology class
\w{[\tilde{k}\sb{n}\circ g\sb{n}]} merely requires that we back-track to an earlier 
stage and try different choices, so in reality we have tree of obstructions, and 
the $G$-interpolation exists if and only if \emph{some} branch extends to infinity.
\end{remark}

\begin{example}\label{eginterp}
We can use the method described here to study $G$-actions on a space $\bY$  
if \w{\pi\sb{i}\bY=0} for \w{i\neq k,m} \wb[:]{k<m} In this case we can choose 
in \w{\bX:=K(\pi,m)} and \w{\bZ:=K(\pi',k)} in \wref[,]{eqinterp} with given 
actions of $G$ on $\pi$ and \w[,]{\pi'} and use Proposition \ref{pinterp} to 
interpolate a $G$-action on $\bY$. As noted in \S \ref{sapplic}, for suitable 
choices of $G$ the obstructions of Proposition \ref{pinterpob} will vanish (e.g., 
for reasons of dimension).
\end{example}

%
%
\sect{Realizing diagrammatic homotopy actions}
\label{cdha}

From now on we assume that $G$ is finite (but see \S \ref{sgen} below). 
Given a $G$-space $\bX$, the associated fixed point set diagram $\uX$ encodes 
the Bredon $G$-homotopy type of $\bX$, by Theorem \ref{telm}. This diagram consists 
of the various fixed point sets \w{\bX\sp{H}} \wb[,]{H\leq G} the inclusions 
\w{i\sp{\ast}:\bX\sp{K}\hra\bX\sp{H}} induced by \w[,]{i:H\hra K} and the 
$G$-action by conjugation: \w[.]{\bX\sp{H}\to\bX\sp{H\sp{a}}} Our goal is to 
provide a ``homotopy version'' of $\uX$, and describe a procedure for realizing 
it by attempting to solve a sequence of simpler lifting problems as in 
Section \ref{ciga}.

\begin{mysubsection}{Filtering $\OG\op$}\label{sfilter}
For any subgroup $H$ of $G$, we define the \emph{length} of $H$ in
$G$, denoted by \w[,]{\lenG H} to be the maximal \w{0\leq k<\infty}
such that there exists a sequence of proper inclusions of subgroups:
\begin{myeq}[\label{eqlength}]
H=H\sb{0}<H\sb{1}<H\sb{2}<\dotsc<H\sb{k-1}<H\sb{k}=G~.
\end{myeq}
\noindent This induces a filtration
\begin{myeq}[\label{eqfilter}]
\F\sb{0}~\subset~\F\sb{1}~\subset~\dotsc\F\sb{k}\subset~\dotsc~\subset~\OG\op
\end{myeq}
\noindent by full subcategories, where
\w{\Obj\F\sb{k}:=\{G/H\in\OG\op~:\ \lenG H\leq k\}} (so
\w[).]{\Obj\F\sb{0}=\{G/G\}}

Since $G$ is finite,  the filtration is exhaustive:
if \w{\lenG\{e\}=N} \wwh that is, the longest possible sequence
\wref{eqlength} in $G$ has $N$ inclusions of proper subgroups \wh then
\w[.]{\F\sb{N}=\OG\op} We let \w{\hF\sb{k}} denote the collection of subgroups 
\w{H<G} such that \w[.]{G/H\in\F\sb{k}}

Let \w{\lra{H}:=\{H\sp{a}~:\ a\in G\}} denote the conjugacy class of a subgroup
\w[:]{H\leq G} note that if \w[,]{H\in\hF\sb{k}} then 
\w[.]{\lra{H}\subseteq\hF\sb{k}}
\end{mysubsection}

\begin{defn}\label{dposet}
Let $\Lambda$ denote the partially ordered set of subgroups of $G$; we can think o
f the opposite category \w{\LG} as a subcategory of \w[.]{\OG} The full 
subcategory \w{\Lambda\sb{H}} consists of all subgroups $K$ with
\w[,]{H<K\leq G} and \w{\Lambda\sb{k}} is the full subcategory of objects in 
filtration \w[.]{\F\sb{k}}
\end{defn}

\begin{defn}\label{dbha}
A \emph{Bredon homotopy action} of $G$
\w{\lra{\wtX,(\Phi\sp{\ast}\sb{H})\sb{\lra{H}\subseteq\Lambda}}} consists of:
\begin{enumerate}
\renewcommand{\labelenumi}{(\roman{enumi})\ }
\item A diagram \w[.]{\wtX:\LG{}\to\TT} 
\item A choice of a representative $H$ in each conjugacy class 
\w[,]{\lra{H}\subseteq\Lambda} equipped with a pointed homotopy \ww{\WH}-action 
\w{\Phi\sp{\ast}\sb{H}:\bB\WH\to\bB\GXp{\tXHH}} on \w[,]{\tXHH} defined by 
the homotopy cofibration sequence:
\begin{myeq}[\label{eqhcofib}]
\tXH~\to~\wtX(H)~\to~\tXHH~,
\end{myeq}
\noindent where \w[.]{\tXH:=\hocolim\sb{\LG\sb{H}}\,\wtX(K)} 
\end{enumerate}
We require that if \w{H'} and $H$ are conjugate, their homotopy cofibration 
sequences \wref{eqhcofib} fit into a homotopy-commuting square \wref[.]{eqhosquare}
\end{defn}

\begin{defn}\label{dcsd}
A cofibrant diagram \w{\uX\sb{k}:\F\sb{k}\to\TT} (in the projective model category 
\w{\TT\sp{\F\sb{k}}} \wwh cf.\ \cite[\S 11.6]{PHirM}) \emph{realizes} a Bredon 
homotopy action \w{\lra{\wtX,(\Phi\sp{\ast}\sb{H})\sb{H\leq G}}} 
\emph{in the $k$-th filtration} if:
\begin{enumerate}
\renewcommand{\labelenumi}{(\alph{enumi})\ }
\item The corresponding homotopy diagram 
\w{(\gamma\circ\uX\sb{k})\rest{\LG\sb{k}}:\LG\sb{k}\to\ho\TT} is weakly equivalent 
to \w[,]{\gamma\circ\wtX\rest{\LG\sb{k}}} for \w{\gamma:\TT\to\ho\TT} the quotient 
functor.
\item For each \w[,]{H\in\F\sb{k}} the pointed action of \w{\WH} on the cofiber
of
\begin{myeq}[\label{eqhcofibre}]
\colim\sb{K>H}\uX\sb{k}(G/K)\to\uX\sb{k}(G/H)
\end{myeq}
\noindent realizes the pointed homotopy action \w[.]{\Phi\sp{\ast}\sb{H}} 
\end{enumerate}

Note that because \w{\uX\sb{k}} is cofibrant, this colimit is a homotopy colimit 
and \wref{eqhcofibre} is a cofibration, and because of \wref[,]{eqhosquare} the 
homotopy action \w{\Phi\sp{\ast}\sb{H}} is defined for \emph{every} \w[,]{H\leq G} 
not only our chosen representatives.

A sequence \w{(\uX\sb{k}:\F\sb{k}\to\TT)\sb{k=0}\sp{\infty}} of such diagrams 
is \emph{coherent} if \w{\uX\sb{k}\rest{\F\sb{k-1}}=\uX\sb{k-1}} for each
\w[.]{k\geq 1}
\end{defn}

\begin{example}\label{egrbha}
If $\bX$ is a $G$-CW complex, let $\wtX$ be the restriction of \w{\uX:\OG\op\to\TT} 
to the subcategory \w[.]{\LG} For any
\w[,]{H\leq G} \w{\tXH:=\hocolim\sb{\LG\sb{H}}~\wtX(K)} is simply 
\w[,]{\bX\sb{H}:=\bigcup\sb{H<K}\,\bX\sp{K}} which is a sub-$\WH$-complex of
\w[.]{\bX\sp{H}} The quotient \w{\XHH:=\bX\sp{H}/\bX\sb{H}} is the cofiber of the 
inclusion \w[,]{j\sb{H}:\bX\sb{H}\hra\bX\sp{H}} which is a free pointed 
\ww{\WH}-space (unless \w[),]{\bX\sb{H}=\emptyset} with monoid action map 
\w{\zeta\sp{\ast}\sb{\XHH}:\WH\to\GXp{\XHH}} and
\w[.]{\Phi\sp{\ast}\sb{H}:=\bB\zeta\sp{\ast}\sb{\XHH}:\bB\WH\to\bB\GXp{\XHH}}
Evidently \w{\uX:\F\sb{\infty}=\OG\op\to\TT} realizes the Bredon homotopy action
\w{\lra{\uX\rest{\LG},(\Phi\sp{\ast}\sb{H})\sb{H\leq G}}} we have just defined 
(in all filtrations). In this case, we also say that the $G$-space $\bX$ 
\emph{realizes} \w[.]{\lra{\uX\rest{\LG},(\Phi\sp{\ast}\sb{H})\sb{H\leq G}}}
\end{example}

\begin{defn}\label{ddiaglift}
If \w{\uX\sb{k}:\F\sb{k}\to\TT} realizes a Bredon homotopy action 
\w{\lra{\wtX,(\Phi\sp{\ast}\sb{H})\sb{H\leq G}}} in the $k$-th filtration, we may 
realize the pointed homotopy action \w{\Phi\sp{\ast}\sb{H}} of \w{\WH} on 
\w{\tXHH} by a topological pointed action of \w{\WH} on a space 
\w{\uXHH\simeq\tXHH} (see Proposition \ref{ppfree} below). 
This fits into a homotopy cofibration sequence:
\begin{myeq}[\label{eqhocofib}]
\ubX{H}~\xra{j\sb{H}}~\wtX(H)~\xra{q\sb{H}}~\uXHH~,
\end{myeq}
\noindent where \w{\ubX{H};=\hocolim\sb{K>H}\uX\sb{k}(G/K)} (homotopic to 
\wref[).]{eqhcofib}
Note that the action of \w{N\sb{G}H} on \w{\F\sb{k}} by conjugation defines 
a \ww{\WH}-action on \w[.]{\ubX{H}}

The $H$-\emph{lifting problem} for \w{\uX\sb{k}} is to find a map 
\w{\Psi\sb{H}} making the following diagram commute up to homotopy:

\mydiagram[\label{eqliftprob}]{
&&&&& \bB\GX{\ubX{H}}\\
\bB\WH \ar[rrrrru]\sp{\bB\zeta\sb{\ubX{H}}} \ar[rrrrrd]\sb{\bB\zeta\sb{\uXHH}} 
\ar@{.>}[rrr]^<<<<<<<<<<<<<<<<<<{\Psi\sb{H}} &&& \bB\QXY{j\sb{H}}{q\sb{H}} 
\ar[rru]\sb{\bB\delta\sb{j\sb{H}}\circ\bB\mu} 
\ar[rrd]\sp{\bB\vre\sb{q\sb{H}}\circ\bB\nu} \\ 
&&&&& \bB\GX{\uXHH}
}
\noindent in the notation of \wref[.]{eqspinterpol}
\end{defn}

We are now in a position to state our main result:

\begin{thm}\label{trha}
A Bredon homotopy action \w{\cA:=\lra{\wtX,(\Phi\sp{\ast}\sb{H})\sb{H\leq G}}} 
for a finite group $G$ can be realized by a $G$-space $\bX$ if and only 
if one can inductively construct a coherent sequence of cofibrant diagrams 
\w{(\uX\sb{k}:\F\sb{k}\to\TT)\sb{k=0}\sp{\infty}} realizing $\cA$, where one can 
extend \w{\uX\sb{k}} to \w{\uX\sb{k+1}} if and only if for each 
\w[,]{\lra{H}\subseteq\hF\sb{k+1}\setminus\hF\sb{k}} there is an \w{H\in\lra{H}} 
for which the $H$-lifting problem \wref{eqliftprob} can be solved.
\end{thm}

\begin{proof}
If $\cA$ can be realized by a $G$-space $\bX$, the corresponding diagrams 
\w{\uX\sb{k}} were described in Example \ref{egrbha}.

To see that solving the $H$-lifting problem suffices to extend an 
inductively-defined \w{\uX\sb{k}} to \w[,]{\uX\sb{k+1}} we start with 
\w{\uX\sb{0}(G/G):=\wtX(G)} (which we denote by $\bY$).
To construct \w[,]{\uX\sb{1}} we must consider all maximal proper subgroups 
\w[,]{M\in\hF\sb{1}} which are of two types:

\begin{enumerate}
\renewcommand{\labelenumi}{(\alph{enumi})\ }
\item If \w[,]{N\sb{G}M=M} then \w{W\sb{M}=\{e\}} and the correspondence 
\w{aM\mapsto M\sp{a}} is a bijection between \w{G/M} and \w[.]{\lra{M}}
In this case we change \w{\bY=\wtX(G/G)\to\wtX(G/M)} into a cofibration 
\w{i:\bY\hra\bZ\sb{(M)}} (with no group action), and form a diagram consisting 
of a copy \w{i\sb{(M\sp{a})}:\bY\hra\bZ\sb{(M\sp{a})}} of $i$ for each coset 
\w[,]{M\sp{a}\in\lra{M}} with \w{\uX\sb{1}(\cp{M}{a})} the homeomorphism 
identifying $\bZ$ with \w{\bZ\sb{(M\sp{a})}} (relative to the fixed subspace $\bY$).
\item Otherwise \w[,]{N\sb{G}M=G} so \w{W\sb{M}=G/M} and \w{\lra{M}} is a singleton. 
We then apply Proposition and \ref{pinterp} to obtain a \ww{W\sb{M}}-action on 
\w[,]{\bZ\sb{(M)}\simeq\wtX(G/M)} extending the trivial action on 
\w[.]{\bY:=\wtX(G/G)} This is possible since we assume that the $M$-lifting problem 
can be solved. The action map \w{\zeta\sb{\bZ\sb{(M)}}:G/M\to\AX{\bZ\sb{(M)}}} 
lifts to a $G$-action via the homomorphism \w[.]{G\epic G/M}
\end{enumerate}

Since all the conjugation $G$-actions we have described agree on $\bY$ (where 
they are trivial), we obtain a diagram \w[,]{\uX\sb{1}:\F\sb{1}\to\TT} whose 
restriction to \w{\LG\sb{1}} consists of the inclusions
\w{\bY\hra\bZ\sb{(M)}} for all \w[.]{M\in\hF\sb{i}\setminus\hF\sb{i-1}}

At the $k$-th stage of the induction, we assume given a cofibrant diagram 
\w{\uX\sb{k-1}:\F\sb{k-1}\to\TT} realizing \w{\wtX} up to filtration \w[.]{k-1} 
In particular, for each \w{H\in\hF\sb{k}\setminus\hF\sb{k-1}} we have a space 
\w{\uX\sb{H}:=(\uX\sb{k-1})\sb{H}} as in \S \ref{dposet}, on which \w{N\sb{G}H} acts 
(by conjugation), with \w{H\subseteq N\sb{G}H} acting trivially. Thus \w{\uX\sb{H}} 
has a \ww{\WH}-action compatible with the structure maps of \w[.]{\uX\sb{k-1}}

For each conjugacy class \w[,]{\lra{H}\subseteq\hF\sb{k}\setminus\hF\sb{k-1}} we 
have a specified representative $H$. We use Proposition \ref{ppfree} to lift 
the given pointed homotopy action of \w{\WH} on \w{\tXHH} to a (free) pointed 
action on \w[.]{\XHH\simeq\tXHH}  Next,  use Proposition \ref{pinterp} to 
produce a \ww{\WH}-interpolation of the given \ww{\WH}-actions on \w{\uX\sb{H}} 
and \w{\XHH} for the homotopy cofibration sequence \wref[.]{eqhcofib} 
Denote the new \ww{\WH}-space we have produced by \w[.]{\bZ\sb{(H)}\simeq\wtX(G/H)} 
By Proposition \ref{pintrp}, we may assume that the inclusion 
\w{i\sb{(H)}:\uX\sb{H}\hra\bZ\sb{(H)}} is \ww{\WH}-equivariant (with respect to the 
given conjugation action on \w[).]{\uX\sb{k-1}}

For any conjugate \w[,]{H\sp{a}\in\lra{H}} choose a fixed element \w{a\in G} 
representing the coset \w[,]{a N\sb{G}H\in G/N\sb{G}H\cong\lra{H}}
and let \w[.]{\uX\sb{k}(G/H\sp{a}):=\bZ\sb{(H)}} The \ww{W\sb{H\sp{a}}}-action on
\w{\uX\sb{k}(G/H\sp{a})} is the composite of the action map \w{\WH\to\AX{\bZ\sb{(H)}}} 
with the isomorphism \w{(\rho\sp{H}\sb{a})\sp{-1}\sb{\ast}:W\sb{H\sp{a}}\to\WH} induced by 
\w{\rho\sp{H}\sb{a}:N\sb{G}H\to N\sb{G}H\sp{a}} (conjugation by $a$).

We define \w{i\sb{H\sp{a}}:\uX\sb{H\sp{a}}\hra\bZ\sb{(H)}} to be the composite
\w[.]{i\sb{(H)}\circ\cp{H}{a}} This is \ww{W\sb{H\sp{a}}}-equivariant because
\w{\cp{H}{a}} is induced by \w[,]{\rho\sp{H}\sb{a}} so we have extended \w{\uX\sb{k-1}} 
to a diagram \w[.]{\uX\sb{k}:\F\sb{k}\to\TT}

At the end of the process we have a full \ww{\OG\op} diagram 
\w[,]{\uX\sb{\infty}:=\colim\sb{k\to\infty}\uX\sb{k}} and thus (by Theorem \ref{telm}) 
a $G$-space $\bX$ realizing the given Bredon homotopy action $\cA$.

Note that homotopic maps \w{\Phi\sp{\ast}\sim(\Phi')\sp{\ast}:\bB\WH\to\bB\GXp{\tXHH}} 
induce pointed Borel \ww{\WH}-equivalences \w{\XHH\to\XpHH} (assuming both 
are \ww{\WH}-CW complexes), and homotopic lifts
\w{\Psi\sim\Psi':\bB\WH\to \bB\QXY{j\sb{H}}{q\sb{H}}} in \wref{eqliftprob}
yield Borel equivalent \ww{\WH}-spaces \w{\bZ\sb{H}} and \w[,]{\bZ'\sb{H}} which 
implies that we have a weak equivalence of the resulting \ww{\F\sb{k+1}}-diagrams 
\w{\uX\sb{k+1}} and\w[,]{\uX'\sb{k+1}} since all the structure maps which are not 
inclusions can be described in terms of the conjugation action of $G$.
\end{proof}

\begin{defn}\label{ddiagob}
If \w{\uX\sb{k}:\F\sb{k}\to\TT} realizes a Bredon homotopy action $\cA$ in the $k$-th 
filtration, for each conjugacy class \w[,]{\lra{H}\subseteq\hF\sb{k+1}\setminus\hF\sb{k}}
choose any representative \w[.]{H\in\lra{H}} The
$\lra{H}$-\emph{sequence of obstructions} \w{(e\sb{n})\sb{n=1}\sp{\infty}} to extending 
\w{\uX\sb{k}} to \w{\uX\sb{k+1}} is defined by letting
\w{e\sb{n}\in H\sp{n+2}(\WH;\,\pi\sb{n+1}F)} denote the $n$-th obstruction of 
Proposition \ref{pinterpob} for the $H$-lifting problem \wref[.]{eqliftprob}

The \emph{difference} obstructions \w{f\sb{n}\in H\sp{n+1}(\WH;\,\pi\sb{n+1}F)} for 
distinguishing between different extensions of \w{\uX\sb{k}} to \w{\uX\sb{k+1}} 
are defined analogously.
\end{defn}

\begin{corollary}\label{crha}
For any finite group $G$, a Bredon homotopy action $\cA$ can be realized 
by a $G$-space $\bX$ if and only if for each \w{k\geq 0} and 
\w[,]{\lra{H}\subseteq\hF\sb{k+1}\setminus\hF\sb{k}} (some branch of) the inductively 
defined \ww{\lra{H}}-sequence of obstructions \w{(e\sb{n})\sb{n=1}\sp{\infty}} vanishes. 
Moreover, two such realizations $\bX$ and \w{\bX'} (by $G$-CW complexes) 
are $G$-homotopy equivalent if the corresponding sequence of difference 
obstructions vanish.
\end{corollary}

Remark \ref{robst} applies here too, of course.

\begin{mysubsection}{Generalizations}\label{sgen}
The procedure described above extends to some infinite groups $G$, as long as we have a class function \w{\ell:\Lambda\to\kappa} into some ordinal $\kappa$ with \w{\ell(K)\preceq\ell(H)} for \w[.]{H\leq K}
In this case we have a filtration corresponding to \wref{eqfilter} of length $\kappa$, and thus a transfinite inductive procedure as in the proof of Theorem \ref{trha}.

For example, if \w{G=\ZZ} then \w{\ell:\Lambda\to\omega+1} assigns to
\w{n\ZZ\leq\ZZ} the number of (not necessarily distinct) prime factors of $n$, with \w[.]{\ell(\{0\})=\omega} On the other hand, there is no such function $\ell$ for \w{G=\ZZ\sp{2}} or \w[.]{S\sp{1}}
\end{mysubsection}

\begin{mysubsection}{Some simple examples}\label{ssexam}
The approach to realizing homotopy actions described here is quite complicated, in general, even for cyclic groups. Nevertheless, in certain cases the theory simplifies to some extent\vsm :

\noindent\textbf{I}.\ \ In a semi-free action all fixed points are global. In terms of a Bredon homotopy action this implies that the maps
\w{j\sb{H}:\tXH\to\wtX(H)} are homotopy equivalences for
\w[,]{\{e\}\neq H} and thus \w{\tXHH} is contractible \wh but \w{\wtX(G)=:\bY} need not be contractible. However, we do have a trivial $G$-action on $\bY$. Thus we are left with the obstructions of Proposition \ref{pinterpob} for interpolating the given $G$-actions in the homotopy cofibration sequence \w[.]{\bY\to\wtX(\{e\})\to\bZ}
If these vanish, we obtain the required semi-free $G$-action an a space \w[.]{\bX\simeq\wtX(\{e\})}

Note that in this case $G$ need not be finite, so the examples mentioned in \S \ref{sapplic} are relevant here\vsm .

\noindent\textbf{II}.\ \ A necessary condition in order for our obstruction theory to be effectively computable is that the (homotopy groups of) the spaces of self-equivalences
of \w{\tXH} and \w{\tXHH} in \wref{eqhcofib} are known for each \w[.]{H\leq G}

One simple case where this holds is when each of the above spaces is an
Eilenberg-Mac~Lane space (cf.\ \wref[).]{eqseem}
If the groups \w{\pi\sb{\ast}\GX{\wtX(H)}} are know \wh e.g., if \w{\wtX(H)} is
also an Eilenberg-Mac~Lane space \wh then the homotopy groups of
\begin{myeq}[\label{eqmsem}]
\map(\tXH,\wtX(H))\hsp \text{and}\hsp \map(\wtX(H),\tXHH)
\end{myeq}
\noindent are known (by \cite{ThoH}), so we may determine the homotopy groups of \w{\PXY{j\sb{H}}} and \w{\PXY{q\sb{H}}} up to an extension
from the pullback diagram \wref[,]{eqspext} \wref[,]{eqseem} and \wref[,]{eqmsem} respectively, and these determine the homotopy groups of
\w{\QXY{j\sb{H}}{q\sb{H}}} \wwh and thus of the fibers $F$ \wh up to extensions from and \wref{eqspinterpol} and \wref[\vsm.]{eqinterpob}

\noindent\textbf{III}.\ \ The discussion above can also be extended to the case of two-stage Postnikov systems (see \S \ref{sapplic} and Example \ref{eginterp} above).
\end{mysubsection}

%
%
\section*{Appendix: \ Pointed homotopy actions}
\label{apha}
\setcounter{thm}{0}
\setcounter{section}{5}

For convenience, we collect here some basic facts about pointed homotopy actions. These are well-known, but we have not found a suitable reference in the literature.

\begin{defn}\label{dpha}
A \emph{pointed homotopy action} of a group $G$ on a pointed space \w{\Xs=(\bX,x\sb{0})} is (the homotopy class of) a map \w[.]{\Phi\sp{\ast}:\BG\to\GXp{\Xs}} It is \emph{realized} by a pointed $G$-action \w{\varphi\sp{\ast}\sb{\Ys}:G\to\AXp{\Ys}} if there is a (pointed) homotopy equivalence \w{h:\Xs\to\Ys} such that

\mydiagram[\label{eqrpha}]{
&&&&&\bB\GXp{\Xs}\ar[d]\sp{\bB h\sb{\ast}} \\
\BG \ar[rrrrru]\sp{\Phi\sp{\ast}} \ar[rrr]\sb{\bB\varphi\sp{\ast}\sb{\Ys}} &&&
\bB\AXp{\Ys} \ar[rr]\sb{\bB i} && \bB\GXp{\Ys}
}
commutes up to homotopy.
\end{defn}

\begin{defn}\label{dphlift}
A $G$-action \w{\varphi:G\to\AX{\bX}} on a space $\bX$
\emph{lifts weakly to a pointed action} \w{\varphi\sp{\ast}:G\to\AXp{\Ys}}
if we have Borel $G$-equivalences \w{p:\bZ\to\bX} and \w[,]{p':\bZ\to\bY} with sections \w{i:\bX\to\bZ} and \w{i':\bY\to\bZ} as in Lemma \ref{lgxequiv},
such that the diagram of associative topological monoids (and multiplicative maps):
\mydiagram[\label{eqtwoact}]{
&&\AXp{\Ys} \ar@{^{(}->}[r] & \GXp{\Ys} \ar@{^{(}->}[r] &
\GX{\bY} \ar[r]\sp{i'\sb{\star}} & \GX{\bZ} \\
G \ar[rru]\sp{\varphi\sp{\ast}\sb{\Ys}} \ar[rr]\sp{\varphi\sb{\bX}} &&
\AX{\bX}\ar@{^{(}->}[rrr] &&& \GX{\bX} \ar[u]\sp{i\sb{\star}}
}
\noindent commutes up to homotopy after inverting the homotopy equivalence \w[.]{i\sb{\star}}

Since we assumed that $\bX$, $\bY$, and $\bZ$  are CW complexes, any homotopy equivalence between them can be made into a pointed homotopy equivalence by choosing appropriate (non-degenerate) base-points (cf.\ \cite[Theorem 3.6]{DoldH}). As a result, we may assume that $\bX$ and $\bY$ in Definition \ref{dphlift} are pointed.
\end{defn}

If \w{\Xs=(\bX,x\sb{0})} is a $G$-space with chosen base point \w[,]{x\sb{0}}  and the $G$-action is free on \w[,]{\bX\setminus\{x\sb{0}\}} we call \w{\Xs} a \emph{free pointed $G$-space}. For any pointed $G$-space, the \emph{associated free pointed $G$-space} is the quotient
$$
\EGXp{G}{\bX}~:=~\EGX{G}{\bX}/\EGX{G}{\{x\sb{0}\}}~,
$$
\noindent  with $G$-action induced from the diagonal action on
\w[.]{\EGX{G}{\bX}} A \emph{homotopy fixed point} for a $G$-space $\bX$ is a
$G$-map \w[.]{f:\EG\to\bX} and we have:

\begin{lemma}\label{lpointed}
Any pointed $G$-space \w{\Xs} has a $G$-map
\w{r:\EGXp{G}{\bX}\to\bX} which is a pointed homotopy equivalence; if
\w{\Xs} is a \emph{free} pointed $G$-space, the map $r$ is a
$G$-homotopy equivalence.
\end{lemma}

\begin{proof}
We have a diagram of $G$-spaces:
\mydiagram[\label{eqmapcofibseq}]{
\EGX{G}{\{x\sb{0}\}}~~\ar@{^{(}->}[rr]\sp{j} \ar[d]\sp{p}\sb{\he} &&
\EGX{G}{\bX} \ar@{->>}[rr]\sp{s} \ar[d]\sp{q}\sb{\he}
\ar@/_{1.5pc}/[ll]\sb{\varphi} &&
\EGXp{G}{\bX} \ar[d]\sp{r} \\
\{x\sb{0}\}~~\ar@{^{(}->}[rr] && \bX \ar@{->}[rr]\sp{\Id} && \bX
}
\noindent where the vertical maps are projections onto the second
factor. Since each row is a cofibration sequence and $p$ and $q$ are
Borel $G$-equivalences, so is $r$.
If the pointed action on \w{\Xs} is free, $r$ induces homotopy
equivalences on all fixed point sets (which consist only of the
basepoint for all \w[),]{\{e\}\neq H\leq G} so by
\cite[Theorem (1.1)]{JSegaE} $r$ is in fact a
$G$-homotopy equivalence.
\end{proof}

\begin{lemma}\label{lfhpt}
A $G$-space $\bX$ with action \w{\varphi:G\to\AX{\bX}} has a homotopy
fixed point corresponding to each weak lift of $\varphi$  to a pointed
action \w[.]{\varphi\sp{\ast}:G\to\AXp{\Ys}}
\end{lemma}

\begin{proof}
A weak lift of $\varphi$ to a pointed action \w{\varphi\sp{\ast}:G\to\AXp{\Ys}} yields a fixed point \w[,]{y\sb{0}\in\bY} and thus a homotopy fixed point for
\w{\Ys=(\bY,y\sb{0})} given by the constant map \w{c\sb{y\sb{0}}:\EG\to\bY}
(which is a $G$-map). This lifts to a homotopy fixed point
\w[.]{\hat{f}:(\Id,c\sb{y\sb{0}}):\EG\to\EGX{G}{\bY}} Since
\w{\Id\times h:\EGX{G}{\bZ}\to\EGX{G}{\bY}} is a $G$-map of free
$G$-CW complexes (\S \ref{rgcw}) which is also a homotopy equivalence, it is actually a $G$-homotopy equivalence by \cite[Theorem (1.1)]{JSegaE}, with
$G$-inverse \w[.]{h\sp{-1}:\EGX{G}{\bY}\to\EGX{G}{\bZ}} The $G$-map
\w{k\circ h\sp{-1}\circ\hat{f}:\EG\to\bX} is the corresponding homotopy
fixed point for $\bX$.

Conversely, if \w{f:\EG\to\bX} is a $G$-map, we may factor $f$ in the
model category \w{\GT} (see \S \ref{sgdiagrams}) as
a $G$-cofibration followed by a $G$-fibration weak equivalence: \w[,.]{\EG\xra{\tilde{f}}\bZ\xepic{p}\bX} If we let
\w{\bY:=\bZ/\EG} denote the (homotopy) cofiber of \w[,]{\tilde{f}}
with quotient $G$-map \w[,]{p':\bZ\to\bY} then $\bY$ has a
basepoint \w{y\sb{0}} (corresponding to \w[),]{\EG\subseteq\bZ} fixed
under the $G$-action, and \w{p'} is a Borel $G$-equivalence since \w{\EG} is contractible. Thus the $G$-action on \w{\Ys=(\bY,y\sb{0})} yields the required pointed lift.
\end{proof}

\begin{prop}\label{ppfree}
Any pointed homotopy action \w{\Phi\sp{\ast}:\BG\to\bB\GXp{\Xs}} can be realized 
by a (free) pointed $G$-action.
\end{prop}

\begin{proof}
Pulling back the universal fibration \w{\bB j} of \wref{eqhfibseq} along 
\w{\Phi:=i'\circ\Phi\sp{\ast}} yields the
following (homotopy) pullback square:
\mydiagram[\label{eqsectionphi}]{
\ar@{}[ddrrrr]|<<<<<<<<<<<<<<<<<<<<<<<<<<<<<
{\framebox{\scriptsize{PB}}}
\BG \ar@{.>}[rrd]\sp{\sigma}
\ar@/_{1.5pc}/[rrdd]\sb{=} \ar@/^{1.5pc}/[rrrrd]\sp{\Phi\sp{\ast}} &&&&\\
&& E\sb{\theta} \ar[rr] \ar[d]\sp{\theta} && \bB\GXp{\Xs}
\ar@{->>}[d]\sp{\bB j}  \\
&& \BG \ar[rr]\sp{\Phi} && \bB\GX{\bX}
}
\noindent so we can use \w{\Phi\sp{\ast}} to obtain a homotopy section
\w{\sigma:\BG\to E\sb{\theta}} as indicated.

We now use the lower left homotopy pullback square in
\wref{eqrowcol} to obtain a homotopy fixed point \w[:]{f:\EG\to\Xp}
\mydiagram[\label{eqfixedpointsig}]{
\ar @{} [ddrrrr] |<<<<<<<<<<<<<<<<<<<<<<<<<<<<<
{\framebox{\scriptsize{PB}}}
\EG \ar@{.>}[rrd]\sp{f}
\ar@/_{1.5pc}/[rrdd]\sb{\sigma\circ q} \ar@/^{1.5pc}/[rrrrd]\sp{=} &&&&\\
&& \Xp \ar[rr] \ar[d] && \EG
\ar@{->>}[d]\sp{q}  \\
&& E\sb{\theta} \ar[rr]\sp{\theta} && \BG~,
}
\noindent where \w[.]{\Xp\simeq\bX} Hence by Lemma \ref{lfhpt} we
obtain a pointed $G$-action on a pointed space \w{\Ys} homotopy
equivalent to $\bX$.
\end{proof}

\end{document}